\documentclass[reqno,a4paper, 11pt]{amsart}

\usepackage[a4paper=true,pdfpagelabels]{hyperref}
\usepackage{graphicx}

\usepackage[ansinew]{inputenc}
\usepackage{amsfonts,epsfig}
\usepackage{latexsym}
\usepackage{mathabx}
\usepackage{amsmath}
\usepackage{amssymb}

\usepackage{color}

\newtheorem{theorem}{Theorem}
\newtheorem{lemma}[theorem]{Lemma}

\newtheorem{proposition}[theorem]{Proposition}

\theoremstyle{definition}

\theoremstyle{remark}

\numberwithin{equation}{section}

\newcommand{\D}{\mathbb{D}}
\newcommand{\DD}{\widehat{\mathcal{D}}}
\newcommand{\Dd}{\widecheck{\mathcal{D}}}

\newcommand{\DDD}{\mathcal{D}}

\newcommand{\N}{\mathbb{N}}

\newcommand{\C}{\mathbb{C}}

\newcommand{\e}{\varepsilon}

\renewcommand{\phi}{\varphi}

\def\a{\alpha}       \def\b{\beta}        
\def\d{\delta}           \def\e{\varepsilon}
     \def\om{\omega}      
\def\s{\sigma}              
\def\p{\psi}         \def\r{\rho}         \def\z{\zeta}
                  \def\vp{\varphi}

\renewcommand{\H}{\mathcal{H}}

\begin{document}

\title{Compact differences of weighted composition operators}


\author{Bin Liu}
\address{University of Eastern Finland, P.O.Box 111, 80101 Joensuu, Finland}
\email{binl@uef.fi}

\author{Jouni R\"atty\"a}
\address{University of Eastern Finland, P.O.Box 111, 80101 Joensuu, Finland}
\email{jouni.rattya@uef.fi}

\thanks{The first author is supported by China Scholarship Council}

\begin{abstract}
Compact differences of two weighted composition operators acting from the weighted Bergman space $A^p_\om$ to another weighted Bergman space $A^q_\nu$, where $0<p\le q<\infty$ and $\om,\nu$ belong to the class~$\DDD$ of radial weights satisfying two-sided doubling conditions, are characterized. On the way to the proof a new description of $q$-Carleson measures for $A^p_\om$, with $\om\in\DDD$, in terms of pseudohyperbolic discs is established. This last-mentioned result generalizes the well-known characterization of $q$-Carleson measures for the classical weighted Bergman space $A^p_\alpha$ with $-1<\alpha<\infty$ to the setting of doubling weights.
\end{abstract}

\maketitle

\section{Introduction and main results}

Let $\H(\D)$ denote the space of analytic functions in the unit disc $\D=\{z\in\C:|z|<1\}$. For a nonnegative function $\om\in L^1([0,1))$, the extension to $\D$, defined by $\om(z)=\om(|z|)$ for all $z\in\D$, is called a radial weight. For $0<p<\infty$ and a radial weight $\omega$, the weighted Bergman space~$A^p_\omega$ consists of $f\in\H(\D)$ such that
    $$
    \|f\|_{A^p_\omega}^p=\int_\D|f(z)|^p\omega(z)\,dA(z)<\infty,
    $$
where $dA(z)=\frac{dx\,dy}{\pi}$ is the normalized Lebesgue area measure on $\D$. As usual,~$A^p_\alpha$ stands for the classical weighted
Bergman space induced by the standard radial weight $\omega(z)=(1-|z|^2)^\alpha$, where $-1<\alpha<\infty$. 

For a radial weight $\om$, write $\widehat{\om}(z)=\int_{|z|}^1\om(s)\,ds$ for all $z\in\D$. In this paper we always assume $\widehat{\om}(z)>0$, for otherwise $A^p_\om=\H(\D)$ for each $0<p<\infty$. A weight $\om$ belongs to the class~$\DD$ if
there exists a constant $C=C(\om)\ge1$ such that $\widehat{\om}(r)\le C\widehat{\om}(\frac{1+r}{2})$ for all $0\le r<1$.
Moreover, if there exist $K=K(\om)>1$ and $C=C(\om)>1$ such that $\widehat{\om}(r)\ge C\widehat{\om}\left(1-\frac{1-r}{K}\right)$ for all $0\le r<1$, then we write $\om\in\Dd$. In other words, $\om\in\Dd$ if there exists $K=K(\om)>1$ and $C'=C'(\om)>0$ such that 
	$$
	\widehat{\om}(r)\le C'\int_r^{1-\frac{1-r}{K}}\om(t)\,dt,\quad 0\le r<1.
	$$
The intersection $\DD\cap\Dd$ is denoted by $\DDD$, and this is the class of weights that we mainly work with.

Each analytic self-map $\vp$ of $\D$ induces the composition operator $C_\varphi$ on $\H(\D)$ defined by $C_\varphi f=f\circ\varphi$. The weighted composition operator induced by $u\in\H(\D)$ and $\vp$ is $u C_\vp$ and sends $f\in\H(\D)$ to $u\cdot f\circ\varphi\in\H(\D)$. These operators have been extensively studied in a variety of function spaces. See for example \cite{CMD,CCM,ZR,ZR1,GTE,SJH1,SJH2,SW,SWY}. XXX

If now $\psi$ is another analytic self-map of $\D$, the pair $(\varphi,\psi)$ induces the operator $C_\varphi-C_\psi$. One of the most important problem considering these operators is to characterize compact differences in Hardy spaces. Shapiro and Sundberg~\cite{SS} studied this problem in 1990. Very recently, Choe, Choi, Koo and Yang~\cite{CB}have solved this problem. For more about difference operators, see \cite{BLW,GTE,HO,MJ,SE}. Moorhouse~\cite{MJ,MT} obtain some important results on this operator in weighted Bergman spaces. He showed~\cite{MJ}, among other things, that $C_\varphi-C_\psi$ is compact on $A^2_\alpha$ if and only if 
      \begin{equation}
      \lim_{|z|\rightarrow 1^-}|\delta_1(z)|\left(\frac{1-|z|^2}{1-|\varphi(z)|^2}+ \frac{1-|z|^2}{1-|\psi(z)|^2}\right)= 0,
      \end{equation}
where
     \begin{equation*}
     \delta_1(z)=\frac{\varphi(z)-\psi(z)}{1-\overline{\varphi(z)}\psi(z)},\quad z\in\D.
     \end{equation*}
Saukko~\cite{SE,SEO} generalized this result by showing that if either $1<p\le q$, or $p>q\ge1$, then $C_\varphi-C_\psi:A^p_\alpha\to A^q_\beta$ is compact if and only if the operators $\delta_1C_\varphi$ and $\delta_1C_\psi$ are both compact from $A^p_\alpha$ to $L^q_\beta$. Very recently, Acharyya and Wu~\cite{AW} characterized the compact differences of two weighted composition operators $uC_\varphi-vC_\psi$ between different weighted Bergman spaces $A^p_\alpha$ and $A^q_\beta$, where $0<p\le q< \infty$ and $-1<\alpha,\beta<\infty$. Their result states that, if $\frac{\alpha+2}{p}\le \frac{\b+2}{q}$ and $u$, $v$ $\in \H(\D)$ satisfy 
	$$
	\sup_{z\in\D}\left(|u(z)|+|v(z)|\right)(1-|z|^2)^{\frac{\b+2}{q}-\frac{\alpha+2}{p}}<\infty,
	$$ 
then  $uC_\varphi-vC_\psi:A^p_\alpha\to A^q_\beta$ is compact if and only if 
       \begin{equation}\label{kusi1}
      \lim_{|z|\rightarrow 1^-}|\delta_1(z)|\left(|u(z)|\frac{(1-|z|^2)^{\frac{2+\beta}{q}}}{(1-|\varphi(z)|^2)^{\frac{2+\alpha}{p}}}+ |v(z)|\frac{(1-|z|^2)^{\frac{2+\beta}{q}}}{(1-|\psi(z)|^2)^{\frac{2+\alpha}{p}}}\right)= 0
      \end{equation}
and
      \begin{equation}\label{kusi2}
      \lim_{|z|\rightarrow 1^-}(1-|\delta_1(z)|^2)^{\frac{2+\alpha}{p}}|u(z)-v(z)|\left(\frac{(1-|z|^2)^{\frac{2+\beta}{q}}}{(1-|\varphi(z)|^2)^{\frac{2+\alpha}{p}}}+ \frac{(1-|z|^2)^{\frac{2+\beta}{q}}}{(1-|\psi(z)|^2)^{\frac{2+\alpha}{p}}}\right)= 0.
      \end{equation}

In this paper we characterize compact differences of two weighted composition operators from the weighted Bergman space $A^p_\omega$ to another weighted Bergman space $A^q_\nu$ with $0<p\le q<\infty$ and $\omega,\nu\in\DDD$. To state the result, write
		\begin{equation*}
    \delta_2(z)=\delta_{2,\vp,\psi}(z)=\frac{\psi(z)-\varphi(z)}{1-\overline{\psi(z)}\varphi(z)},\quad z\in\D,
    \end{equation*}
and observe that $|\d_1|=|\d_2|$ on $\D$. Our main result reads as follows.  

\begin{theorem}\label{theorem:0 theorem}
Let $\om,\nu\in\DDD$ and $0<p\le q<\infty$ such that $\widehat{\nu}(z)(1-|z|)\lesssim\left(\widehat{\om}(z)(1-|z|)\right)^{\frac{q}{p}}$ for all $z\in\D$. Further, let $u,v\in\H(\D)$ and $\varphi$ and $\psi$ analytic self-maps of $\D$ such that
	\begin{equation}\label{Eq:sufficiency-hypothesis-1}
	\sup_{z\in\D}\left(|u(z)|+|v(z)|\right)\frac{\left(\widehat{\nu}(z)(1-|z|)\right)^{\frac{1}{q}}}{\left(\widehat{\omega}(z)(1-|z|)\right)^{\frac{1}{p}}}<\infty.
	\end{equation}
Then there exists $\gamma=\gamma(\om,p)>0$ with the following property: $uC_\varphi-vC_\psi:A^p_\om\rightarrow A^q_\nu$ is compact if and only if
      \begin{equation}\label{Eq:sufficiency-hypothesis-3}
      \lim_{|z|\rightarrow 1^-}|\delta_1(z)|
			\left(|u(z)|\frac{\left(\widehat{\nu}(z)(1-|z|)\right)^{\frac{1}{q}}}
			{\left(\widehat{\om}(\varphi(z))(1-|\varphi(z)|)\right)^{\frac{1}{p}}}
			+|v(z)|\frac{\left(\widehat{\nu}(z)(1-|z|)\right)^{\frac{1}{q}}}
			{\left(\widehat{\om}(\psi(z))(1-|\psi(z)|)\right)^{\frac{1}{p}}}\right)= 0
      \end{equation}
and     
			
			\begin{equation}\label{Eq:sufficiency-hypothesis-4}
			\begin{split}
			\lim_{|z|\rightarrow 1^-}&\Bigg(\frac{|1-\overline{\vp(z)}\delta_1(z)|^\gamma}
			{\left(\widehat{\om}(\varphi(z))(1-|\varphi(z)|)\right)^{\frac{1}{p}}}
			+\frac{|1-\overline{\p(z)}\d_2(z)|^\gamma}{\left(\widehat{\om}(\psi(z))(1-|\psi(z)|)\right)^{\frac{1}{p}}}\Bigg)\\
			&\cdot|u(z)-v(z)|\left(\widehat{\nu}(z)(1-|z|)\right)^{\frac{1}{q}}=0.
      \end{split}
			\end{equation}
\end{theorem}

If $\om(z)=(1-|z|)^{\alpha}$ and $\nu(z)=(1-|z|)^\beta$ for $-1<\alpha,\beta<\infty$, then $\widehat{\om}(z)\asymp(1-|z|)^{\alpha+1}$ and $\widehat{\nu}(z)\asymp(1-|z|)^{\beta+1}$ for all $z\in\D$. Therefore \eqref{Eq:sufficiency-hypothesis-3} reduces to \eqref{kusi1}. Moreover, the proof of Theorem~\ref{theorem:0 theorem} shows that the only requirement for $\gamma=\gamma(\om,p)>0$ appearing in the statement is that 
	$$
	\int_\D\frac{\om(z)}{|1-\overline{a}z|^{\gamma p}}\,dA(z)\le C\frac{\widehat{\om}(a)}{(1-|a|)^{\gamma p-1}},\quad a\in\D,
	$$
for some constant $C=C(\om,p,\gamma)>0$. If $\om(z)=(1-|z|)^{\alpha}$, any $\gamma>\frac{\alpha+2}{p}$ is acceptable, and the choice $\gamma=2\frac{\alpha+2}{p}$ converts \eqref{Eq:sufficiency-hypothesis-4} to \eqref{kusi2}, as a simple computation shows. Therefore Theorem~\ref{theorem:0 theorem} indeed generalizes \cite[Theorem~1]{AW} for weights in $\DDD$.

We need two specific tools for the proof of Theorem~\ref{theorem:0 theorem}. The first one concerns continuous embeddings $A^p_\om\subset L^q_\mu$. Recall that a positive Borel measure $\mu$ on $\D$ is a $q$-Carleson measure for $A^p_\om$ if the identity operator $I_d:A^p_\om\to L^q_\mu$ is bounded. A complete characterization of such measures in the case $\om\in\DD$ can be found in \cite{PelRatEmb}, see also \cite{PR,PelSum14}. In particular, it is known that if $q\ge p$ and $\om\in\DD$, then $\mu$ is a $q$-Carleson measure for $A^p_\om$ if and only if
	$$
	\sup_{a\in\D}\frac{\mu(S(a))}{\om(S(a))^\frac{q}{p}}<\infty.
	$$
Here and from now on $S(a)=\{z:1-|a|<|z|<1,\,|\arg z-\arg a|<(1-|a|)/2\}$ is the Carleson square induced by the point $a\in\D\setminus\{0\}$, $S(0)=\D$ and $\om(E)=\int_E\om dA$ for each measurable set $E\subset\D$. We will need a variant of this result and its ``compact'' counterpart for $\om\in\DDD$ where the Carleson squares are replaced by pseudohyperbolic discs. To this end, denote $\varphi_a(z)=\frac{a-z}{1-\overline{a}z}$ for $a,z\in\D$. The pseudohyperbolic distance between two points $a$ and $b$ in $\D$ is $\rho(a,b)=|\varphi_a(b)|$. For $a\in\D$ and $0<r<1$, the pseudohyperbolic disc of center $a$ and of radius $r$ is $\Delta(a,r)=\{z\in \D:\rho(a,z)<r\}$. It is well known that $\Delta(a,r)$ is an Euclidean disk centered at $(1-r^2)a/(1-r^2|a|^2)$ and of radius $(1-|a|^2)r/(1-r^2|a|^2)$.

\begin{theorem}\label{Lemma:Carleson}
Let $0<p\le q<\infty$, $\om\in\DDD$ and $\mu$ a positive Borel measure on $\D$. Then there exists $r=r(\om)\in(0,1)$ such that the following statements hold:
\begin{itemize}
\item[\rm(i)] $\mu$ is a $q$-Carleson measure for $A^p_\omega$ if
and only if
    \begin{equation}\label{eq:s1}
    \sup_{a\in\D}\frac{\mu\left(\Delta(a,r)\right)}{\left(\om\left(\Delta(a,r)\right)\right)^\frac{q}p}<\infty.
    \end{equation}
Moreover, if $\mu$ is a $q$-Carleson measure for $A^p_\omega$,
then the identity operator satisfies
    $$
    \|I_d\|^q_{A^p_\om\to L^q_\mu}
		\asymp\sup_{a\in\D}\frac{\mu\left(\Delta(a,r)\right)}{\left(\om\left(\Delta(a,r)\right)\right)^\frac{q}p}.
    $$
\item[\rm(ii)] The identity operator $I_d:A^p_{\om}\to L^q_\mu$ is compact if and only if
\begin{equation}\label{eq:s1compact}
    \lim_{|a|\to1^-}\frac{\mu\left(\Delta(a,r)\right)}{\left(\om\left(\Delta(a,r)\right)\right)^\frac{q}p}=0.
    \end{equation}
    \end{itemize}
\end{theorem}

Another result needed is a lemma that allows us to estimate the distance between images of two points, say $z$ and $a$, under $f$ sufficiently accurately whenever $z$ is close to $a$ in the sense that $z\in\Delta(a,r)$, and $f\in A^p_\om$ with $\om\in\DDD$. For the statement, denote $\widetilde{\om}(z)=\widehat{\om}(z)/(1-|z|)$ for all $z\in\D$. 

\begin{lemma}\label{le1}
Let $\om\in\DDD$, $0<p\le q<\infty$ and $0<r<R<1$. Then there exists a constant $C=C(\om,p,q,r,R)>0$ such that
     \begin{equation}\label{Eq:lemma}
     |f(z)-f(a)|^q\le C\frac{\rho(z,a)^q}{\left(\widehat{\omega}(a)(1-|a|)\right)^{\frac{q}{p}}}
		\int_{\Delta(a,R)}|f(\zeta)|^p\widetilde{\om}(\zeta)\,dA(\zeta), \quad a\in \D,\quad z\in\Delta(a,r), 
     \end{equation}
for all $f\in A^p_\om$ with $\|f\|_{A^p_\om}\le1$.
\end{lemma}

This lemma plays an important role in the proof of Theorem~\ref{theorem:0 theorem} when we show that \eqref{Eq:sufficiency-hypothesis-3} and \eqref{Eq:sufficiency-hypothesis-4} are sufficient conditions for the compactness. By \cite[Proposition~5]{JJK} we know that 
	\begin{equation}\label{eq:equivalent norms}
	\|f\|_{A^p_{\widetilde{\om}}}\asymp\|f\|_{A^p_\om},\quad f\in\H(\D),
	\end{equation}
provided $\om\in\DDD$. This explains the appearance of the weight $\widetilde{\om}$ on the right hand side of \eqref{Eq:lemma}. It is worth observing that, despite of \eqref{eq:equivalent norms}, the strictly positive weight $\widetilde{\om}$ cannot be replaced by $\om$ in the statement because $\omega\in\DDD$ may vanish in pseudohyperbolic discs of fixed radius that tend to the boundary.

The rest of the paper contains the proofs of the results stated above. We first prove Lemma~\ref{le1} in the next section. The proof of the result on Carleson measures, Theorem~\ref{Lemma:Carleson}, is given in Section~\ref{Sec:Carleson}, and finally, Theorem~\ref{theorem:0 theorem} is proved in Section~\ref{Sec:proof-of-main}.

To this end, couple of words about the notation used in the sequel. The letter $C=C(\cdot)$ will denote an absolute constant whose value depends on the parameters indicated
in the parenthesis, and may change from one occurrence to another.
We will use the notation $a\lesssim b$ if there exists a constant
$C=C(\cdot)>0$ such that $a\le Cb$, and $a\gtrsim b$ is understood
in an analogous manner. In particular, if $a\lesssim b$ and
$a\gtrsim b$, then we write $a\asymp b$ and say that $a$ and $b$ are comparable.

\section{Proof of Lemma~\ref{le1}}

It is known that if $\om\in\DDD$, then there exist constants $0<\alpha=\alpha(\om)\le\b=\b(\om)<\infty$ and $C=C(\om)\ge1$ such that
	\begin{equation}\label{Eq:characterization-D}
	\frac1C\left(\frac{1-r}{1-t}\right)^\alpha
	\le\frac{\widehat{\om}(r)}{\widehat{\om}(t)}
	\le C\left(\frac{1-r}{1-t}\right)^\beta,\quad 0\le r\le t<1.
	\end{equation}
In fact, this pair of inequalities characterizes the class $\DDD$ because the right hand inequality is satisfied if and only if $\om\in\DD$ by \cite[Lemma~2.1]{PelSum14} while the left hand inequality describes the class $\Dd$ in an analogous way, see \cite[(2.27)]{PelRat2020}. The chain of inequalities \eqref{Eq:characterization-D} will be frequently used in the sequel.

To prove the lemma, let $a\in\D$, $0<r<1$ and $z\in\Delta(a,r)$. Then 
      \begin{equation}\label{Eq:1}
			\begin{split}
      |f(z)-f(a)|^p
			&=|f(\varphi_a(\varphi_a(z)))-f(\varphi_a(0))|^p
			=\left|\int_0^{\varphi_a(z)}(f\circ\varphi_a)'(\xi)\,d\xi\right|^p\\
      &\le\max_{\xi\in \overline{D(0,|\varphi_a(z)|)}}|(f\circ\varphi_a)'(\xi)|^p|\varphi_a(z)|^p
			\le\max_{\xi\in \overline{D(0,r)}}|(f\circ\varphi_a)'(\xi)|^p|\varphi_a(z)|^p.
      \end{split}
			\end{equation}
Let $R\in(r,1)$ and set $R'=\frac{r+R}{2}$. Further, let $0<s<1$. Then the Cauchy integral formula for the derivative and the subharmonicity of $|f|^p$ yield
      \begin{equation}\label{Eq:2}
			\begin{split}
      |(f\circ\varphi_a)'(\xi)|^p
			&=\left|\frac{1}{2\pi}\int_{|w|=R'}\frac{(f\circ\varphi_a)(w)}{(w-\xi)^2}\,dw\right|^p
      \le\left(\frac{2}{R-r}\right)^{2p}\left(R'\right)^p\max_{|w|=R'}|f(\varphi_a(w))|^p\\
      &\lesssim\frac{4}{\pi s^2} \max_{|w|=R'}\frac{1}{(1-|\varphi_a(w)|)^2}\int_{\Delta(\varphi_a(w),s)}|f(\zeta)|^p\,dA(\zeta)\\
      &\lesssim\max_{|w|=R'}\frac{1}{(1-|a|)^2}\int_{\Delta(\varphi_a(w),s)}|f(\zeta)|^p\,dA(\zeta),\quad \xi\in\overline{D(0,r)}.
      \end{split}
			\end{equation}
Fix now $s=s(r,R)\in(0,1)$ sufficiently small such that $\Delta(\varphi_a(w),s)\subset\Delta(a,R)$ for all $w$ such that $|w|=R'$. Further, an application of the right hand inequality in \eqref{Eq:characterization-D} shows that $\widehat{\om}(\zeta)\asymp\widehat{\om}(a)$ for all $\zeta\in\Delta(a,R)$. Therefore, by combining \eqref{Eq:1} and \eqref{Eq:2} we deduce 
      \begin{equation*}
			\begin{split}
      |f(z)-f(a)|^p
			&\lesssim\frac{|\varphi_a(z)|^p}{(1-|a|)^2}\int_{\Delta(a,R)}|f(\zeta)|^p\,dA(\zeta)\\
      &\lesssim\frac{|\varphi_a(z)|^p}{\widehat{\om}(a)(1-|a|)}\int_{\Delta(a,R)}|f(\zeta)|^p~\widetilde{\om}(\zeta)\,dA(\zeta),
			\quad a\in\D,\quad z\in\Delta(a,r).
      \end{split}
			\end{equation*}
This proves the case $p=q$ because $|\varphi_a(z)|=\rho(z,a)$ for all $a,z\in\D$. This part of the proof is valid for all $f\in\H(\D)$ if $\om\in\DD$.

Let now $q>p$, and observe that trivially $|f(z)-f(a)|^q=\left(|f(z)-f(a)|^p\right)^{\frac{q}{p}}$. An application of the case $q=p$ implies      
      \begin{equation*}
			\begin{split}
      |f(z)-f(a)|^q
			&\lesssim\frac{\rho(z,a)^q}{\left(\widehat{\omega}(a)(1-|a|)\right)^{\frac{q}{p}}}
			\left(\int_{\Delta(a,R)}|f(\zeta)|^p\widetilde{\om}(\zeta)\,dA(\zeta)\right)^{\frac{q}{p}}\\
			&\le\frac{\rho(z,a)^q\|f\|_{A^p_{\widetilde{\om}}}^{q-p}}{\left(\widehat{\omega}(a)(1-|a|)\right)^{\frac{q}{p}}}
			\int_{\Delta(a,R)}|f(\zeta)|^p\widetilde{\om}(\zeta)\,dA(\zeta).
      \end{split}
			\end{equation*}
But \eqref{eq:equivalent norms} guarantees $\|f\|_{A^p_{\widetilde{\om}}}\asymp\|f\|_{A^p_\om}\le1$, and thus the assertion in the case $q>p$ follows from the above estimate.

\section{Proof of Theorem~\ref{Lemma:Carleson}}\label{Sec:Carleson}

To prove (i), assume first \eqref{eq:s1} and let $0<r<1$. The fact that $|f|^p$ is subharmonic in~$\D$ together with Minkowski's inequality in continuous form (Fubini's theorem in the case $q=p$) and \eqref{eq:s1} imply
	\begin{equation*}
	\begin{split}
	\|f\|_{L^q_\mu}^q&\lesssim\int_\D\left(\int_{\Delta(z,r)}\frac{|f(\zeta)|^p}{(1-|\z|)^2}\,dA(\z)\right)^\frac{q}{p}\,d\mu(z)
	\le\left(\int_\D|f(\zeta)|^p\frac{\mu(\Delta(\zeta,r))^\frac{p}{q}}{(1-|\z|)^2}\,dA(\zeta)\right)^\frac{q}{p}\\
	&\lesssim\left(\int_\D|f(\zeta)|^p\frac{\om(\Delta(\zeta,r))}{(1-|\z|)^2}\,dA(\zeta)\right)^\frac{q}{p},\quad f\in\H(\D).
	\end{split}
	\end{equation*}
Since $\om\in\DDD$ by the hypothesis, we may apply the right hand inequality in \eqref{Eq:characterization-D} to deduce 
	\begin{equation}\label{1111}
	\om(\Delta(\zeta,r))\lesssim\widehat{\om}(\z)(1-|\z|),\quad \z\in\D.
	\end{equation}
It follows that $\|f\|_{L^q_\mu}\lesssim\|f\|_{A^p_{\widetilde{\om}}}$, and hence $\|f\|_{L^q_\mu}\lesssim\|f\|_{A^p_{\om}}$ for all $f\in\H(\D)$ by \eqref{eq:equivalent norms}. Thus $\mu$ is a $q$-Carleson measure $A^p_\om$. 

Conversely, assume that $\mu$ is a $q$-Carleson measure $A^p_\om$. For each $a\in\D$, consider the function 
	\begin{equation}\label{eq:testfunctions}
	f_{a}(z)=\left(\frac{1-|a|^2}{1-\overline{a}z}\right)^\gamma\frac{1}{\left(\widehat{\om}(a)(1-|a|)\right)^{\frac{1}{p}}}
	=\frac{\left(1-\overline{a}\vp_a(z)\right)^\gamma}{\left(\widehat{\om}(a)(1-|a|)\right)^{\frac{1}{p}}},\quad z\in\D,
	\end{equation}
induced by $\om$ and $0<\gamma,p<\infty$. Then \cite[Lemma~2.1]{PelSum14} implies that for all $\gamma=\gamma(\om,p)>0$ sufficiently large we have $\|f_{a}\|_{A^p_\om}\asymp1$ for all $a\in\D$. Therefore the assumption yields
	$$
	1\asymp\|f_a\|^q_{A^p_\om}\gtrsim\|f_a\|_{L^q_\mu}^q\gtrsim\frac{\mu(\Delta(a,r))}{(\widehat{\om}(a)(1-|a|))^\frac{q}{p}},\quad a\in\D,
	$$
that is, $\mu(\Delta(a,r))\lesssim(\widehat{\om}(a)(1-|a|))^\frac{q}{p}$ for all $a\in\D$. Since $\om\in\DDD\subset\Dd$ by the hypothesis, there exists $K=K(\om)>1$ and $C=C(\om)>1$ such that $\widehat{\om}(r)\ge C\widehat{\om}\left(1-\frac{1-r}{K}\right)$ for all $0\le r<1$ by the definition. Fix now $r=r(K)\in(0,1)$ sufficiently large such that 
	$$
	\Delta(a,r)\supset\left\{te^{i\theta}:|a|\le t \le1-\frac{1-|a|}{K},\,\,|\arg a-\theta|\le\frac{1-|a|}{2}\left(1-\frac1K\right)\right\}.
	$$
Then, as $\om\in\DDD\subset\DD$, the right hand inequality in \eqref{Eq:characterization-D} yields
	\begin{equation*}
	\begin{split}
	\om(\Delta(a,r))
	&\ge(1-|a|)\left(1-\frac1K\right)|a|\int_{|a|}^{1-\frac{1-|a|}{K}}\om(s)\,ds\\
	&\ge(C-1)(1-|a|)\left(1-\frac1K\right)|a|\widehat{\om}\left(1-\frac{1-|a|}{K}\right)\\
	&\gtrsim\widehat{\om}(a)(1-|a|)|a|,\quad a\in\D,
	\end{split}
	\end{equation*}
and therefore
	$$
	\mu(\Delta(a,r))\lesssim(\widehat{\om}(a)(1-|a|))^\frac{q}{p}
	\lesssim\left(\frac{\om(\Delta(a,r))}{|a|}\right)^\frac{q}{p},\quad a\in\D\setminus\{0\}.
	$$
The claim \eqref{eq:s1} now follows from these estimates for all $r=r(\om)\in(0,1)$ sufficiently large.

To prove (ii), assume first that $I_d:A^p_\om\to L^q_\mu$ is compact. An application of \cite[Lemma~2.1]{PelSum14} and the right hand inequality in \eqref{Eq:characterization-D} ensure that we may choose $\gamma=\gamma(p,\omega)>0$ sufficiently large such that $\|f_{a}\|_{A^p_\om}\asymp1$ for all $a\in\D$, and $f_{a}\to0$ uniformly on compact subsets of $\D$, as $|a|\to1^-$. Therefore the closure of the set $\{f_{a}:a\in\D\}$ is compact in $L^q_\mu$. Since for each $\e>0$ the open balls $B(f_a,\e)=\{f\in L^q_\mu:\|f_a-f\|_{L^q_\mu}<\e\}$ cover $\overline{\{f_{a}:a\in\D\}}$, there exists a finite subcover $\left\{B(f_{a_n},\e):n=1,\ldots,N=N(\e)\right\}$. Let now $a\in\D$ be arbitrary, and let $j=j(a)\in\{1,\ldots,N\}$ such that $f_a\in B(f_{a_n},\e)$. Then, for each $R\in(0,1)$, we have
	\begin{equation*}
	\begin{split}
	\int_{\D\setminus D(0,R)}|f_{a}(z)|^q\,d\mu(z)
	&\lesssim\int_{\D\setminus D(0,R)}|f_{a}(z)-f_{a_j}(z)|^q\,d\mu(z)
	+\int_{\D\setminus D(0,R)}|f_{a_j}(z)|^q\,d\mu(z)\\
	&\le\|f_a-f_{a_j}\|_{L^q_\mu}^q
	+\max_{n=1,\ldots,N}\int_{\D\setminus D(0,R)}|f_{a_n}(z)|^q\,d\mu(z).
	\end{split}
	\end{equation*}
By fixing $R\in(0,1)$ sufficiently large, and taking into account that $\e>0$ was arbitrary, we deduce
	\begin{equation*}
	\lim_{R\to1^-}\int_{\D\setminus D(0,R)}|f_{a}(z)|^q\,d\mu(z)=0
	\end{equation*}
uniformly in $a$. This together with the uniform convergence yield
	$$
	0=\lim_{|a|\to1^-}\|f_{a}\|_{L^q_\mu}^q
	\ge\lim_{|a|\to1^-}\int_{\Delta(a,r)}|f_{a}(z)|^q\,d\mu(z)
	\gtrsim\lim_{|a|\to1^-}\frac{\mu(\Delta(a,r))}{(\widehat{\om}(a)(1-|a|))^\frac{q}{p}}
	$$
for all $r\in(0,1)$. Now fix $r=r(\om)$ as in the case (i) to have $\widehat{\om}(a)(1-|a|)\lesssim\om(\Delta(a,r))$ for all $a\in\D$. Then we obtain \eqref{eq:s1compact}.

Conversely, assume \eqref{eq:s1compact}. Let $\{f_k\}_{k\in\N}$ be a sequence in $A^p_\om$ such that $\sup_{k\in\N}\|f_k\|_{A^p_\om}=M<\infty$. Then it is easy to see that $\{f_k\}_{k\in\N}$ is uniformly bounded on compact subsets of $\D$ -- this follows, for example, from \eqref{Eq:radial-growth} below. Therefore $\{f_k\}_{k\in\N}$ constitutes a normal family by Montel's theorem, and hence we may extract a subsequence $\{f_{k_j}\}_{j\in\N}$ that converges uniformly on compact subsets of $\D$ to a function $f$ which belongs to $\H(\D)$ by Weierstrass' theorem. Fatou's lemma now shows that $f\in A^p_\om$. For $r\in(0,1)$, fix an $r$-lattice $\{a_n\}_{n\in\N}$. Since $|a_n|\to1$, as $n\to\infty$, we have
	$$
	\lim_{n\to\infty}\frac{\mu(\Delta(a_n,r))}{\om(\Delta(a_n,r))^\frac{q}{p}}=0
	$$
by the hypothesis. Therefore, for each $\e>0$, there exists $N=N(\e)\in\N$ such that 
	$$
	\frac{\mu(\Delta(a_n,r))}{\om(\Delta(a_n,r))^\frac{q}{p}}<\e,\quad n\ge N.
	$$
Hence, as in the case (i), Minkowski's inequality in continuous form (Fubini's theorem in the case $q=p$), \eqref{1111}, \eqref{Eq:characterization-D} and \eqref{eq:equivalent norms} yield
	\begin{equation*}
	\begin{split}
	&\sum_{n=N}^\infty\int_{\Delta(a_n,r)}|f(z)-f_{k_j}(z)|^q\,d\mu(z)\\
	&\quad\lesssim\sum_{n=N}^\infty\int_{\Delta(a_n,r)}\left(\int_{\Delta(z,R)}\frac{|f(\z)-f_{k_j}(\z)|^p}{(1-|\z|)^2}\,dA(\z)\right)^\frac{q}{p}\,d\mu(z)\\
	&\quad\le\sum_{n=N}^\infty\left(\int_{\{\z:\Delta(a_n,r)\cap\Delta(\z,R)\ne\emptyset\}}\frac{|f(\z)-f_{k_j}(\z)|^p}{(1-|\z|)^2}
	\mu\left(\Delta(a_n,r)\right)^\frac{p}{q}dA(\z)\right)^\frac{q}{p}\\
	&\quad\le\e\sum_{n=N}^\infty\left(\int_{\{\z:\Delta(a_n,r)\cap\Delta(\z,R)\ne\emptyset\}}\frac{|f(\z)-f_{k_j}(\z)|^p}{(1-|\z|)^2}
	\om\left(\Delta(a_n,r)\right)dA(\z)\right)^\frac{q}{p}\\
	&\quad\lesssim\e\left(\sum_{n=N}^\infty\int_{\{\z:\Delta(a_n,r)\cap\Delta(\z,R)\ne\emptyset\}}|f(\z)-f_{k_j}(\z)|^p\widetilde{\om}(\z)\,dA(\z)\right)^\frac{q}{p}\\
	&\quad\lesssim\e\|f-f_{k_j}\|^q_{A^p_{\widetilde{\om}}}
	\asymp\e\|f-f_{k_j}\|^q_{A^p_{\om}}
	\lesssim M^q\e.
	\end{split}
	\end{equation*}
Since
	$$
	\lim_{j\to\infty}\sum_{n=1}^{N-1}\int_{\Delta(a_n,r)}|f(z)-f_{k_j}(z)|^q\,d\mu(z)=0
	$$
by the uniform convergence in compact subsets, we deduce
	\begin{equation*}
	\begin{split}
	\limsup_{j\to\infty}\int_\D|f(z)-f_{k_j}(z)|^q\,d\mu(z)
	&\le\limsup_{j\to\infty}\Bigg(\sum_{n=1}^{N-1}\int_{\Delta(a_n,r)}|f(z)-f_{k_j}(z)|^q\,d\mu(z)\\
	&\qquad\qquad+\sum_{n=N}^{\infty}\int_{\Delta(a_n,r)}|f(z)-f_{k_j}(z)|^q\,d\mu(z)\Bigg)\lesssim\e.
	\end{split}
	\end{equation*}
Since $\e>0$ was arbitrary, we have
	$$
	\limsup_{j\to\infty}\int_\D|f(z)-f_{k_j}(z)|^q\,d\mu(z)=0,
	$$
and hence $I_d:A^p_\om\to L^q_\mu$ is compact. This completes the proof of the theorem.

\section{Proof of Theorem~\ref{theorem:0 theorem}}\label{Sec:proof-of-main}

With the auxiliary results proved in the previous sections we are ready for the proof of the main result. We will follow the arguments used in~\cite{AW} with appropriate modifications. The following two propositions will prove Theorem~\ref{theorem:0 theorem}. The first one gives necessary conditions for $uC_\varphi-vC_\psi:A^p_\om\rightarrow A^q_\nu$ to be compact.

\begin{proposition}\label{theorem:1 theorem}
Let $\om,\nu\in\DDD$,  $0<p, q<\infty$, $u,v\in\H(\D)$ and $\varphi$ and $\psi$ be analytic self-maps of $\D$. If $uC_\varphi-vC_\psi:A^p_\om\rightarrow A^q_\nu$ is compact, then
      \begin{equation*}
      \lim_{|z|\rightarrow 1^-}|\delta_1(z)|\left(|u(z)|\frac{\left(\widehat{\nu}(z)(1-|z|)\right)^{\frac{1}{q}}}{\left(\widehat{\om}(\varphi(z))(1-|\varphi(z)|)\right)^{\frac{1}{p}}}+ |v(z)|\frac{\left(\widehat{\nu}(z)(1-|z|)\right)^{\frac{1}{q}}}{\left(\widehat{\om}(\psi(z))(1-|\psi(z)|)\right)^{\frac{1}{p}}}\right)= 0
      \end{equation*}
and there exits $\gamma=\gamma(\om,p)>0$ such that     
     \begin{equation*}
      \lim_{|z|\rightarrow 1^-}\Bigg(\frac{|1-\overline{\vp}(z)\delta_1(z)|^\gamma}{\left(\widehat{\om}(\varphi(z))(1-|\varphi(z)|)\right)^{\frac{1}{p}}}
	  +\frac{|1-\overline{\p}(z)\delta_2(z)|^\gamma}{\left(\widehat{\om}(\psi(z))(1-|\psi(z)|)\right)^{\frac{1}{p}}}\Bigg)|u(z)-v(z)|\left(\widehat{\nu}(z)(1-|z|)\right)^{\frac{1}{q}}=0.
      \end{equation*}
\end{proposition}

\begin{proof}
Consider the test functions $f_a$ defined in \eqref{eq:testfunctions}, and set $F_{a}(z)=\varphi_a(z)f_{a}(z)$ for all $a,z\in\D$. Obviously, $\|F_{a}\|_{A^p_\omega}\le\|f_{a}\|_{A^p_\om}$ for all $a\in\D$. Further, by the proof of Theorem~\ref{Lemma:Carleson}, both $f_{a}$ and $F_{a}$ tend to zero uniformly on compact subsets of $\D$ as $|a|\rightarrow 1^-$, and $\|f_{a}\|_{A^p_\om}\asymp1$ for all $a\in\D$ if $\gamma=\gamma(\om,p)>0$ is sufficiently large. Since $uC_\varphi-vC_\psi:A^p_\om \rightarrow A^q_\nu$ is compact by the hypothesis, we therefore have
     \begin{equation}\label{Eq:comp-testing-1}
     \lim_{|a|\rightarrow 1^-}\|uC_\varphi(f_{a})-vC_\psi(f_{a})\|_{A^q_\nu}=0
     \end{equation}
and
		 \begin{equation}\label{Eq:comp-testing-2}
     \lim_{|a|\rightarrow 1^-}\|uC_\varphi(F_{a})-vC_\psi(F_{a})\|_{A^q_\nu}=0.
		 \end{equation}
Also, if $\lim_{|a|\rightarrow 1^-}$ is replaced by $\sup_{a\in\D}$ in the above formulas, then the corresponding quantities are bounded.

We next observe that for each $\om\in\DD$ and $0<q<\infty$, there exists a positive bounded function $C_\om$ on $[0,1)$ such that
	\begin{equation}\label{Eq:radial-growth}
  |f(z)|\le\frac{C_\om(|z|)}{\left(\widehat{\om}(z)(1-|z|)\right)^\frac{1}{q}}\|f\|_{A^q_\om},\quad f\in A^q_\om,\quad z\in\D,
  \end{equation}
and $C_\om(z)\to0$ as $|z|\to1^-$. Namely, for each $z\in\D$ we have
	\begin{equation*}
	\begin{split}
	\|f\|_{A^q_\om}^q
	&\ge\int_{\D\setminus D(0,\frac{1+|z|}{2})}|f(\zeta)|^q\om(\zeta)\,dA(\zeta)
	\ge M_q^q\left(\frac{1+|z|}{2},f\right)\int_{\frac{1+|z|}{2}}^1r\om(r)\,dr\\
	&\ge\frac12M_q^q\left(\frac{1+|z|}{2},f\right)\widehat\om\left(\frac{1+|z|}{2}\right)
	\gtrsim M_q^q\left(\frac{1+|z|}{2},f\right)\widehat\om\left(z\right),
	\end{split}
	\end{equation*}
which combined with the well-known inequality $M_\infty(|z|,f)\lesssim M_q\left(\frac{1+|z|}{2},f\right)(1-|z|)^{-\frac1q}$, valid for all $f\in\H(\D)$, yields \eqref{Eq:radial-growth} because 
	$$
	\int_{\D\setminus D(0,\frac{1+|z|}{2})}|f(\zeta)|^q\om(\zeta)\,dA(\zeta)\to0,\quad |z|\to1^-,
	$$
for each $f\in A^q_\om$.

By combining \eqref{Eq:comp-testing-1} and \eqref{Eq:comp-testing-2} with \eqref{Eq:radial-growth}, we deduce
      \begin{equation*}
      \lim_{\max(|a|,|z|)\rightarrow 1^-} |uC_\varphi(f_{a})(z)-vC_\psi(f_{a})(z)|\left(\widehat{\nu}(z)(1-|z|)\right)^{\frac{1}{q}}=0
      \end{equation*}
and
			\begin{equation*}
			\lim_{\max(|a|,|z|)\rightarrow 1^-} |uC_\varphi(F_{a})(z)-vC_\psi(F_{a})(z)|\left(\widehat{\nu}(z)(1-|z|)\right)^{\frac{1}{q}}=0.
      \end{equation*}
By choosing $a=\varphi(z)$, we obtain
      \begin{equation}\label{Eq:comp-testing-1-evo:1}
      \lim_{|z|\rightarrow 1^-} 
			|u(z)f_{\varphi(z)}(\varphi(z))-v(z)f_{\varphi(z)}(\psi(z))|\left(\widehat{\nu}(z)(1-|z|)\right)^{\frac{1}{q}}=0
			\end{equation}
and
			\begin{equation}\label{Eq:comp-testing-2-evo:1}
      \lim_{|z|\rightarrow 1^-} 
			|\delta_1(z)||v(z)||f_{\varphi(z)}(\psi(z))|\left(\widehat{\nu}(z)(1-|z|)\right)^{\frac{1}{q}}=0.
      \end{equation}             
Since $|\delta_1(z)|<1$ for all $z\in\D$, and
      \begin{equation*}
      |u(z)||f_{\varphi(z)}(\varphi(z))|
			\le|u(z)f_{\varphi(z)}(\varphi(z))-v(z)f_{\varphi(z)}(\psi(z))|
			+|v(z)||f_{\varphi(z)}(\psi(z))|,
      \end{equation*}
by combining \eqref{Eq:comp-testing-1-evo:1} and \eqref{Eq:comp-testing-2-evo:1} we deduce
      \begin{equation}\label{pili}
			\begin{split}
      &\lim_{|z|\rightarrow 1^-}\frac{|\delta_1(z)||u(z)|\left(\widehat{\nu}(z)(1-|z|)\right)^{\frac{1}{q}}}{\left(\widehat{\omega}(\varphi(z))(1-|\varphi(z)|)\right)^{\frac{1}{p}}}\\
			&\quad= \lim_{|z|\rightarrow 1^-}|\delta_1(z)||u(z)||f_{\varphi(z)}(\varphi(z))|\left(\widehat{\nu}(z)(1-|z|)\right)^{\frac{1}{q}}=0.
      \end{split}
			\end{equation}
Further, we claim that for each $0<\gamma<\infty$ and each bounded set $\Omega\subset\C$, there exists a constant $C=C(\gamma,\Omega)>0$ such that 
	\begin{equation}\label{Eq:puskalojo}
	|1-z^\gamma|\le C|1-z|,\quad z\in\Omega,
	\end{equation}
the proof of which is postponed for a moment. By using this and the fact that $1-\frac{1-|a|^2}{1-\overline{a}b}=\overline{a}\vp_a(b)$ for all $a,b\in\D$, we deduce
       \begin{equation*}
			\begin{split}
      |f_{\varphi(z)}(\varphi(z))-f_{\varphi(z)}(\psi(z))|
			&=|f_{\varphi(z)}(\varphi(z))|\left|1-\left(\frac{1-|\varphi(z)|^2}{1-\overline{\varphi(z)}\psi(z)}\right)^\gamma\right|\\
      &\lesssim|f_{\varphi(z)}(\varphi(z))|\left|1-\frac{1-|\varphi(z)|^2}{1-\overline{\varphi(z)}\psi(z)}\right|
      \le|f_{\varphi(z)}(\varphi(z))||\delta_1(z)|,
      \end{split}
			\end{equation*}
and hence
      \begin{equation*}
			\begin{split}
      |u(z)-v(z)||f_{\varphi(z)}(\psi(z))|
			&\le|u(z)f_{\varphi(z)}(\varphi(z))-v(z)f_{\varphi(z)}(\psi(z))|\\
      &\quad+|u(z)||f_{\varphi(z)}(\varphi(z))-f_{\varphi(z)}(\psi(z))|\\
      &\lesssim|u(z)f_{\varphi(z)}(\varphi(z))-v(z)f_{\varphi(z)}(\psi(z))|\\
      &\quad+|\delta_1(z)||u(z)||f_{\varphi(z)}(\varphi(z))|,\quad z\in\D.
      \end{split}
			\end{equation*}
Therefore, by \eqref{Eq:comp-testing-1-evo:1} and \eqref{pili}, we finally obtain
       \begin{equation*}
       \begin{aligned}
       &\lim_{|z|\rightarrow 1^-}\Bigg(\frac{|1-\overline{\vp}(z)\delta_1(z)|^\gamma}{\left(\widehat{\om}(\varphi(z))(1-|\varphi(z)|)\right)^{\frac{1}{p}}}\Bigg)|u(z)-v(z)|\left(\widehat{\nu}(z)(1-|z|)\right)^{\frac{1}{q}}\\
       &=\lim_{|z|\rightarrow 1^-}|u(z)-v(z)||f_{\varphi(z)}(\psi(z))|\left(\widehat{\nu}(z)(1-|z|)\right)^{\frac{1}{q}}=0.
       \end{aligned}
       \end{equation*}   
By following the reasoning above, but with the choice $a=\psi(z)$, we obtain 
       $$
			\lim_{|z|\rightarrow 1^-}\frac{|\delta_1(z)||v(z)|\left(\widehat{\nu}(z)(1-|z|)\right)^{\frac{1}{q}}}{\left(\widehat{\omega}(\psi(z))(1-|\psi(z)|)\right)^{\frac{1}{p}}}=0
			$$
as an analogue of \eqref{pili}, and then eventually
      $$
			\lim_{|z|\rightarrow 1^-}\Bigg(\frac{|1-\overline{\p}(z)\d_2(z)|^\gamma}
			{\left(\widehat{\om}(\psi(z))(1-|\psi(z)|)\right)^{\frac{1}{p}}}\Bigg)|u(z)-v(z)|\left(\widehat{\nu}(z)(1-|z|)\right)^{\frac{1}{q}}=0.
			$$
Therefore to finish the proof of the proposition, it remains to establish \eqref{Eq:puskalojo}. If $z\in\Omega\setminus\{z:|1-z|<1/2\}$, then
	$$
	|1-z^\gamma|\le 1+\sup_{z\in\Omega}|z|^\gamma\le2\left(1+\sup_{z\in\Omega}|z|^\gamma\right)|1-z|,
	$$
while if $z\in\Omega\cap\{z:|1-z|<1/2\}$, we have
	$$
	|1-z^\gamma|=\left|\int_z^1\gamma\zeta^{\gamma-1}\,d\zeta\right|
	\le\gamma\int_z^1|\z|^{\gamma-1}|d\zeta|
	\le\frac{\gamma\max\{1,3^{\gamma-1}\}}{2^{\gamma-1}}|1-z|.
	$$
This proves \eqref{Eq:puskalojo}, and completes the proof of the proposition.
\end{proof}      

Sufficient conditions for the compactness of $uC_\varphi-vC_\psi:A^p_\om\rightarrow A^q_\nu$ are given in the next result.

\begin{proposition}\label{theorem:2 theorem}
Let $\om,\nu\in\DDD$ and $0<p\le q<\infty$ such that $\widehat{\nu}(z)(1-|z|)\lesssim\left(\widehat{\om}(z)(1-|z|)\right)^{\frac{q}{p}}$ for all $z\in\D$. Further, let $u,v\in\H(\D)$ and $\varphi$ and $\psi$ analytic self-maps of $\D$ such that \eqref{Eq:sufficiency-hypothesis-1} is satisfied. If there exists $\gamma>0$ such that \eqref{Eq:sufficiency-hypothesis-3} and \eqref{Eq:sufficiency-hypothesis-4} are satisfied, then $uC_\varphi-vC_\psi:A^p_\om\rightarrow A^q_\nu$ is compact.
\end{proposition}

\begin{proof}
It suffices to show that for any norm bounded sequence $\{f_n\}$ in $A^p_\om$ which tends to zero uniformly on compact subsets of $\D$ as $n\rightarrow\infty$, we have $\|(uC_\varphi-vC_\psi)(f_n)\|_{A^q_\nu}\rightarrow0$ as $n\rightarrow\infty$. For simplicity, assume $\|f_n\|_{A^p_\om}\le1$ for all $n$. Fix $0<r<R<1$, and denote $E=\{z\in\D:|\delta_1(z)|<r\}$ and $E'=\D\setminus E$. Write  
      \begin{equation*}
      (uC_\varphi-vC_\psi)(f_n)=(uC_\varphi-vC_\psi)(f_n)\chi_{E'}+(u-v)C_\psi(f_n)\chi_E+u(C_\varphi-C_\psi)(f_n)\chi_E,
      \end{equation*}
and observe that it is enough to prove that each of the three quantities
      \begin{equation}\label{Eq:3-conditions}
      \|(uC_\varphi-vC_\psi)(f_n)\chi_{E'}\|_{ A^q_\nu},\quad 
			\|(u-v)C_\psi(f_n)\chi_E\|_{ A^q_\nu}\quad\textrm{and}\quad
			\|u(C_\varphi-C_\psi)(f_n)\chi_E\|_{ A^q_\nu}
      \end{equation}
tends to zero as $n\rightarrow\infty$.

We begin with considering the first two quantities in \eqref{Eq:3-conditions}. By the definition of the set $E$ we have the estimates
      \begin{equation*}
      |(uC_\varphi-vC_\psi)(f_n)\chi_{E'}|\le\frac{1}{r}(|\delta_1 uC_\varphi(f_n)|+|\delta_1vC_\psi(f_n)|)
			\end{equation*}
and
			\begin{equation*}
      |(u-v)C_\psi(f_n)\chi_E|\le\left(\frac{1}{1-r}\right)^\gamma|1-\overline{\psi}\delta_2|^\gamma|u-v||C_\psi(f_n)|
      \end{equation*}
on $\D$. Therefore it suffices to prove that $\delta_1uC_\varphi$, $\delta_1vC_\psi$ and $(1-\overline{\p}\d_2)^\gamma(u-v)C_\p$ are compact operators from $A^p_\om$ to $L^q_\nu$. We show in detail that $\delta_1uC_\varphi$ is compact - the same argument shows the compactness of the other two operators.

Let $\mu$ be a finite nonnegative Borel measure on $\D$ and $h$ a measureable function on $\D$. For an analytic self-map $\varphi$ of $\D$, the weighted pushforward measure is defined by 
       \begin{equation}\label{Eq:pusforward}
       \varphi_*(h,\mu)(M)=\int_{\varphi^{-1}(M)}h d\mu
       \end{equation}
for each measurable set $M\subset\D$. If $\mu$ is the Lebesgue measure, we omit the measure in the notation and write $\varphi_*(h)(M)$ for the left hand side of \eqref{Eq:pusforward}. By the measure theoretic change of variable~\cite[Section~39]{PM}, we have $\|\delta_1uC_\varphi(f)\|_{L^q_\nu}=\|f\|_{L^q_{\varphi_*(|\delta_1u|^q\nu)}}$ for each $f\in A^p_\om$. Therefore Theorem~\ref{Lemma:Carleson} shows that $\delta_1 uC_\varphi:A^p_\om\rightarrow L^q_\nu$ is compact if and only if 
	$$
	\frac{\varphi_*(|\delta_1u|^q\nu)(\Delta(a,r))}{\om(\Delta(a,r))^\frac{q}{p}}
	=\frac{\int_{\vp^{-1}(\Delta(a,r))}|\delta_1(z)u(z)|^q\nu(z)\,dA(z)}{\om(\Delta(a,r))^\frac{q}{p}}\to0,\quad |a|\to1^-.
	$$
This is what we prove next. Define 
	$$
	W_{a,r}=\sup_{z\in\varphi^{-1}(\Delta(a,r))}\left|\left(\delta_1(z)u(z)\right)^q\frac{\widehat{\nu}(z)(1-|z|)}{\left(\widehat{\omega}(\varphi(z))(1-|\varphi(z)|)\right)^{\frac{q}{p}}}\right|.
	$$
Then $W_{a,r}\rightarrow 0$, as $|a|\rightarrow 1^-$, by the hypothesis \eqref{Eq:sufficiency-hypothesis-3}. Moreover, for $a\in \D$ and $z\in \varphi^{-1}(\Delta(a,r))$, \eqref{Eq:characterization-D} yields
        \begin{equation*}
        \left|\delta_1(z)u(z)\right|^q
				\lesssim W_{a,r}\frac{\left(\widehat{\omega}(a)(1-|a|)\right)^{\frac{q}{p}}}{\widehat{\nu}(z)(1-|z|)},
        \end{equation*}
and therefore, for each $\e\in(0,1)$ we have 
      \begin{equation}\label{kuli}
			\begin{split}
      \varphi_*(|\delta_1u|^q\nu)(\Delta(a,r))
			&=\int_{\varphi^{-1}(\Delta(a,r))}\left|\delta_1(z)u(z)\right|^q\nu(z)\,dA(z)\\
      &\lesssim\sup_{z\in \D}\left|\left(\delta_1(z)u(z)\right)^q \frac{\widehat{\nu}(z)(1-|z|)}
			{\left(\widehat{\omega}(z)(1-|z|)\right)^{\frac{q}{p}}}\right|^{1-\e}
			W_{a,r}^\e \left(\widehat{\omega}(a)(1-|a|)\right)^{{\e\frac{q}{p}}} \\
      &\quad\cdot\int_{\varphi^{-1}(\Delta(a,r))}\frac{\left(\widehat{\omega}(z)(1-|z|)\right)^{(1-\e)\frac{q}{p}}\nu(z)}
			{\widehat{\nu}(z)(1-|z|)}dA(z),\quad a\in\D.
			\end{split}
      \end{equation}

Before proceeding further, we indicate how to get to this point with the operator $(1-\overline{\p}\d_2)^\gamma(u-v)C_\p$. After the measure theoretic change of variable and an application of Theorem~\ref{Lemma:Carleson}, consider 
	$$
	V_{a,r}= \sup_{z\in \p^{-1}(\Delta(a,r))}
	\left|\left(1-\overline{\p}(z)\d_2(z)\right)^{\gamma q} \left(u(z)-v(z)\right)^q\frac{\widehat{\nu}(z)(1-|z|)}{\left(\widehat{\omega}(\p(z))(1-|\p(z)|)\right)^{\frac{q}{p}}}\right|
	$$ 
instead of $W_{a,r}$. Then $V_{a,r}\rightarrow 0$, as $|a|\rightarrow 1^-$, by the hypothesis \eqref{Eq:sufficiency-hypothesis-4}, and moreover, for $a\in \D$ and $z\in \p^{-1}(\Delta(a,r))$, \eqref{Eq:characterization-D} yields
        \begin{equation*}
        \left|1-\overline{\p}(z)\d_2(z)\right|^{\gamma q} \left|u(z)-v(z)\right|^q
				\lesssim V_{a,r}\frac{\left(\widehat{\omega}(a)(1-|a|)\right)^{\frac{q}{p}}}{\widehat{\nu}(z)(1-|z|)}.
        \end{equation*}
Therefore, for each $\e\in(0,1)$, we have
     \begin{equation*}
			\begin{split}
      &\p_*(\left|1-\overline{\p}\d_2\right|^{\gamma q}|u-v|^q\nu)(\Delta(a,r))\\
      &\lesssim\sup_{z\in \D}
			\left|\left(1-\overline{\p}(z)\d_2(z)\right)^{\gamma q} \left(u(z)-v(z)\right)^q \frac{\widehat{\nu}(z)(1-|z|)}
			{\left(\widehat{\omega}(z)(1-|z|)\right)^{\frac{q}{p}}}\right|^{1-\e}
			V_{a,r}^\e \left(\widehat{\omega}(a)(1-|a|)\right)^{{\e\frac{q}{p}}} \\
      &\quad\cdot\int_{\p^{-1}(\Delta(a,r))}\frac{\left(\widehat{\omega}(z)(1-|z|)\right)^{(1-\e)\frac{q}{p}}\nu(z)}
			{\widehat{\nu}(z)(1-|z|)}dA(z),\quad a\in\D.
			\end{split}
      \end{equation*}  

To estimate this last integral, which is the same as the one appearing in \eqref{kuli}, we first show that there exists $\e=\e(\om,\nu,q,p)\in(0,1)$ sufficiently small such that the function
	$$
	\mu(z)=\mu_{\om,\nu,\e,q,p}(z)
	=\frac{\left(\widehat{\omega}(z)(1-|z|)\right)^{(1-\e)\frac{q}{p}}\nu(z)}{\widehat{\nu}(z)(1-|z|)},\quad z\in\D,
	$$
is a weight and belongs to $\DDD$. To see that $\mu\in\DD$, for $n\in\N\cup\{0\}$, define $\rho_n$ by $\widehat{\nu}(\r_n)=\frac{\widehat{\nu}(0)}{K^n}$, where $K>1$. Let $0\le r<1$ and fix $M\in\N\cup\{0\}$ such that $\r_M\le r<\r_{M+1}$. Set $\alpha=(1-\e)\frac{q}{p}$ for short. Since $\om,\nu\in\DDD$ by the hypothesis, \eqref{Eq:characterization-D} yields
	\begin{equation*}
	\begin{split}
	\widehat{\om}(\r_{j+N+1})
	&\le\widehat{\om}(\r_{j})
	\lesssim\widehat{\om}(\r_{j+N+1})\left(\frac{1-\r_j}{1-\r_{j+N+1}}\right)^{\beta(\omega)}
	\lesssim\widehat{\om}(\r_{j+N+1})\left(\frac{\widehat{\nu}(\r_{j})}{\widehat{\nu}(\r_{j+N+1})}\right)^{\frac{\beta(\omega)}{\alpha(\nu)}}\\
	&=\widehat{\om}(\r_{j+N+1})K^{(N+1)\frac{\beta(\omega)}{\alpha(\nu)}}
	\asymp\widehat{\om}(\r_{j+N+1})
	\end{split}
	\end{equation*}
and $1-\r_j\asymp1-\r_{j+1}$ for all $j$ and for each fixed $N\in\N$. Therefore 
      \begin{equation*}
			\begin{split}
      \int_r^1\frac{(\widehat{\om}(t))^\alpha}{(1-t)^{1-\alpha}}\frac{\nu(t)}{\widehat{\nu}(t)}dt
			&\le\sum_{j=M}^\infty\int_{\r_j}^{\r_{j+1}}\frac{(\widehat{\om}(t))^\alpha}{(1-t)^{1-\alpha}}\frac{\nu(t)}{\widehat{\nu}(t)}\,dt
      \le\sum_{j=M}^\infty\frac{(\widehat{\om}(\r_j))^\alpha}{(1-\r_{j+1})^{1-\alpha}}
			\int_{\r_{j}}^{\r_{j+1}}\frac{\nu(t)}{\widehat{\nu}(t)}dt\\
			&=\sum_{j=M}^\infty\frac{(\widehat{\om}(\r_j))^\alpha}{(1-\r_{j+1})^{1-\alpha}}
			\int_{\r_{j+N}}^{\r_{j+N+1}}\frac{\nu(t)}{\widehat{\nu}(t)}dt
      \asymp\int_{\r_{M+N}}^1\frac{(\widehat{\om}(t))^\alpha}{(1-t)^{1-\alpha}}\frac{\nu(t)}{\widehat{\nu}(t)}\,dt.
       \end{split}
			\end{equation*}
Another application of \eqref{Eq:characterization-D} shows that there exists $N=N(\nu)\in\mathbb{N}$ such that $\r_{M+N}\ge\frac{1+\r_{M+1}}{2}$. Namely, the right hand inequality implies
	$$
	\frac{1-\r_{M+1}}{1-\r_{M+N}}\ge\left(\frac{\widehat{\nu}(\r_{M+1})}{C\widehat{\nu}(\r_{M+N})}\right)^\frac1\b
	=\left(\frac{K^{N-1}}{C}\right)^\frac1{\b}\ge2
	$$
for sufficiently large $N$ giving what we want. Therefore 
      \begin{equation*}
			\begin{split}
      \int_r^1\frac{(\widehat{\om}(t))^\alpha}{(1-t)^{1-\alpha}}\frac{\nu(t)}{\widehat{\nu}(t)}dt
      \lesssim \int_{{\frac{1+\r_{M+1}}{2}}}^1\frac{(\widehat{\om}(t))^\alpha}{(1-t)^{1-\alpha}}\frac{\nu(t)}{\widehat{\nu}(t)}dt 
			\le\int_{{\frac{1+r}{2}}}^1\frac{(\widehat{\om}(t))^\alpha}{(1-t)^{1-\alpha}}\frac{\nu(t)}{\widehat{\nu}(t)}dt,
      \end{split}
			\end{equation*}
which shows that $\mu\in\DD$, provided $\mu$ is a weight. 

We next show that $\mu$ is a weight in $\Dd$. By using \eqref{Eq:characterization-D}, with $\b$ in place of $\a$, we obtain       
		\begin{equation*}
		\begin{split}
    \int_r^1\frac{(\widehat{\om}(t))^\alpha}{(1-t)^{1-\alpha}}\frac{\nu(t)}{\widehat{\nu}(t)}dt
		&\lesssim\int_r^1 \frac{\left(\widehat{\om}(r)\left(\frac{1-t}{1-r}\right)^\b(1-t)\right)^\a}{\widehat{\nu}(t)(1-t)}\nu(t)\,dt\\
		&\lesssim\frac{\widehat{\om}(r)^\a}{(1-r)^{\b\a}} \int_r^1 \frac{(1-t)^{{(1+\b)\a}}}{\widehat{\nu}(t)(1-t)}\nu(t)\,dt.
    \end{split}
		\end{equation*}
Now fix $\e\in(0,1)$ sufficient small such that $\s=(1+\b)\a=(1+\b)(1-\e)\frac{q}{p}>1$. By \cite[Lemma 3]{JAJ}, the last expression above is dominated by a constant times
    \begin{equation*}
		\begin{split}
    \frac{\widehat{\om}(r)^\a}{(1-r)^{\b\a}}(1-r)^{\s-1}
		&=\frac{\left(\widehat{\om}(r)(1-r)\right)^\a}{1-r}.
		\end{split}
		\end{equation*}
Since $\nu\in\Dd$ by the hypothesis, there exists $K=K(\nu)>1$ such that		
		\begin{equation*}
		\begin{split}
    \frac{\left(\widehat{\om}(r)(1-r)\right)^\a}{1-r} 
		&\asymp \frac{\left(\widehat{\om}(r)(1-r)\right)^\a}{\widehat{\nu}(r)(1-r)}\int_r^{1-\frac{1-r}{K}}\nu(t)dt\\
		&\asymp\int_r^{1-\frac{1-r}{K}}\frac{\left(\widehat{\om}(t)(1-t)\right)^\a}{\widehat{\nu}(t)(1-t)} \nu(t)dt,
    \end{split}
		\end{equation*}
where the last step is a consequence of \eqref{Eq:characterization-D}, applied to both weights $\om,\nu\in\DDD$. This reasoning shows that $\mu$ is a weight in $\Dd$, and thus $\mu\in\DDD$.

We return to estimate the last integral in \eqref{kuli}. By \cite[Proposition~18]{PelRatSchatten}, the operator $C_\vp:A^p_\mu\rightarrow A^p_\mu$ is bounded for each $0<p<\infty$. By the measure theoretic change of variable, this is equivalent to saying that $I_d:A^p_\mu\rightarrow L^p_{\vp_\ast(\mu)}$ is bounded. Since we just proved that $\mu\in\DDD$, this is in turn equivalent to $\vp_\ast(\mu)(\Delta(a,r))\lesssim\mu(\Delta(a,r))$ by Theorem~\ref{Lemma:Carleson}. By the definition of these two measures and the poof of the fact $\mu\in\Dd$ above, we have
	$$
	\int_{\varphi^{-1}(\Delta(a,r))}\mu(z)\,dA(z)
	\lesssim\int_{\Delta(a,r)}\mu(z)\,dA(z) 
	\lesssim \left(\widehat{\omega}(a)(1-|a|)\right)^{{(1-\e)\frac{q}{p}}},\quad a\in\D.
	$$
This combined with \eqref{kuli} gives
       \begin{equation*}
       \varphi_*(|\delta_1u|^q\nu)(\Delta(a,r))
			\lesssim\sup_{z\in\D}\left|\left(\delta_1(z)u(z)\right)^q \frac{\widehat{\nu}(z)(1-|z|)}{\left(\widehat{\omega}(z)(1-|z|)\right)^{\frac{q}{p}}}\right|^{1-\e} W_{a,r}^\e \left(\widehat{\omega}(a)(1-|a|)\right)^{{\frac{q}{p}}}.
       \end{equation*}
Since the supremum above is bounded by the hypothesis \eqref{Eq:sufficiency-hypothesis-1}, and $\widehat{\omega}(a)(1-|a|)\lesssim\om(\Delta(a,r))$ for $r=r(\om)\in(0,1)$ sufficiently large by the proof of Theorem~\ref{Lemma:Carleson}, we deduce via Theorem~\ref{Lemma:Carleson} that $\delta_1uC_\varphi:A^p_\om\rightarrow L^q_\nu$ is compact. As mentioned already, $\delta_1vC_\psi$ and $(1-\overline{\psi}\delta_2)^\gamma(u-v)C_\psi$ can be treated in the same way.

It remains to deal with the third term in \eqref{Eq:3-conditions}. By Lemma \ref{le1}, Fubini's theorem and \eqref{Eq:characterization-D}, we have 
      \begin{equation*}
			\begin{split}
      \|u(C_\varphi-C_\psi)(f_n)\chi_E\|^q_{A^q_\nu}
			&=\int_E|u(z)|^q|f_n(\varphi(z))-f_n(\psi(z))|^q\nu(z)\,dA(z)\\
      &\lesssim\int_E\frac{|u(z)\delta_1(z)|^q}{\left(\widehat{\omega}(\varphi(z))(1-|\varphi(z)|)\right)^{\frac{q}{p}}}
			\int_{\Delta(\varphi(z),R)}|f_n(\z)|^p\widetilde{\om}(\z)\,dA(\z)\nu(z)\,dA(z)\\
       &\le\int_{\D}|f_n(\z)|^p\widetilde{\omega}(\z)\\
			&\quad\cdot\left(\int_{\varphi^{-1}(\Delta(\z,R))\cap E}\frac{|u(z)\delta_1(z)|^q}{\left(\widehat{\omega}(\varphi(z))(1-|\varphi(z)|)\right)^{\frac{q}{p}}}\nu(z)\,dA(z)\right)dA(\z)\\
       &\asymp\int_\D|f_n(\z)|^p\widetilde{\om}(\z)
			\left(\int_{\varphi^{-1}(\Delta(\z,R))}\frac{|u(z)\delta_1(z)|^q}{\left(\widehat{\omega}(\z)(1-|\z|)\right)^{\frac{q}{p}}}\nu(z)\,dA(z)\right)dA(\z)\\
       &=\int_\D|f_n(\z)|^p\frac{\varphi_*(|\delta_1 u|^q\nu)(\Delta(\z,R))}{\left(\widehat{\omega}(\z)(1-|\z|)\right)^{\frac{q}{p}}}\widetilde{\omega}(\z)\,dA(\z).
       \end{split}
			\end{equation*}
Since the identity operator from $A^p_\om$ to $L^q_{\varphi_*(|\delta_1u|^q\nu)}$ is compact, it is also bounded. This and Theorem~\ref{Lemma:Carleson} yield
        \begin{equation*}
				\begin{split}
        \|u(C_\varphi-C_\psi)(f_n)\chi_E\|^q_{A^q_\nu}
				&\lesssim\sup_{\z\in D(0,r)}\frac{\varphi_*(|\delta_1u|^q\nu)(\Delta(\z,R))}{\left(\widehat{\omega}(\z)(1-|\z|)\right)^{\frac{q}{p}}}
				\int_{D(0,r)}|f_n(\z)|^p\widetilde{\omega}(\z)\,dA(\z)\\
        &\quad+\sup_{\z\in\D\setminus D(0,r)}\frac{\varphi_*(|\delta_1u|^q\nu)(\Delta(\z,R))}{\left(\widehat{\omega}(\z)(1-|\z|)\right)^{\frac{q}{p}}} \int_{\D\setminus D(0,r)}|f_n(\z)|^p\widetilde{\omega}(\z)\,dA(\z)\\
         &\lesssim\sup_{\z\in D(0,r)}|f_n(\z)|^p + \sup_{\z\in\D\setminus D(0,r)}\frac{\varphi_*(|\delta_1u|^q\nu)(\Delta(\z,R))}{\left(\widehat{\omega}(\z)(1-|\z|)\right)^{\frac{q}{p}}},\quad 0<r<1.
         \end{split}
				\end{equation*}
By choosing $0<r<1$ sufficiently large, the last term can be made smaller than a pregiven $\e>0$. For such fixed $r$, the first term tends to zero as $n\to\infty$ by the uniform convergence. Therefore
	$$
	\|u(C_\varphi-C_\psi)(f_n)\chi_E\|^q_{A^q_\nu}\to0,\quad n\to\infty,
	$$
and hence also the last term in \eqref{Eq:3-conditions} tends to zero. This finishes the proof of the proposition.
\end{proof}


\begin{thebibliography}{99}

\bibitem{AW}            Acharyya, Soumyadip; Wu, Zhijian: Compact and 
                        Hilbert-Schmidt differences of weighted composition 
                        operators. Integral Equations Operator Theory 88 
                        (2017), no. 4, 465--482.

\bibitem{BLW}           Bonet, Jos\'e; Lindstr\"om, Mikael; Wolf, Elke: 
                        Differences of composition operators between weighted 
                        Banach spaces of holomorphic functions. J. Aust. 
                        Math. Soc. 84 (2008), no. 1, 9--20.

\bibitem{CB}            Choe, Boo Rim; Choi, Koeun; Koo, Hyungwoon; Yang, 
                        Jongho: Difference of weighted composition operators. J. 
                        Funct. Anal. 278 (2020), no. 5, 108401, 38 pp.

\bibitem{CMD}           Contreras, Manuel D.; Hern\'andez-D\'iaz, Alfredo G.:  
                        Weighted composition operators on Hardy spaces. J. 
                        Math. Anal. Appl. 263 (2001), no. 1, 224--233.

\bibitem{CCM}           Cowen, Carl C.; MacCluer, Barbara D.: Composition 
                        operators on spaces of analytic functions. Studies in 
                        Advanced Mathematics. CRC Press, Boca Raton, FL, 
                        1995. xii+388 pp.

\bibitem{ZR}            \v{C}u\v{c}kovi\'c, \v{Z}eljko; Zhao, Ruhan: Weighted 
                        composition operators between different weighted 
                        Bergman spaces and different Hardy spaces. Illinois 
                        J. Math. 51 (2007), no. 2, 479--498.

\bibitem{ZR1}           \v{C}u\v{c}kovi\'c, \v{Z}eljko; Zhao, Ruhan: Weighted 
                        composition operators on the Bergman space. J. London 
                        Math. Soc. (2) 70 (2004), no. 2, 499--511.

\bibitem{GTE}           Goebeler, Thomas E., Jr.: Composition operators acting 
                        between Hardy spaces. Integral Equations Operator Theory 
                        41 (2001), no. 4, 389--395. 

\bibitem{PM}            Halmos, Paul R. : Measure Theory, Springer-Verlag, New York, 1974.

\bibitem{HO}            Hosokawa, Takuya; Ohno, Sh\^{u}ichi: Differences of 
                        composition operators on the Bloch spaces. J. 
                        Operator Theory 57 (2007), no. 2, 229--242. 

\bibitem{MJ}            Moorhouse, Jennifer: Compact differences of 
                        composition operators. J. Funct. Anal. 219 (2005), 
                        no. 1, 70--92.

\bibitem{MT}            Moorhouse, Jennifer; Toews, Carl: Differences of 
                        composition operators. Trends in Banach spaces and 
                        operator theory (Memphis, TN, 2001), 207--213,     
                        Contemp. Math., 321, Amer. Math. Soc., Providence, 
                        RI, 2003. 

\bibitem{PR}            Pel\'aez, Jos\'e \'Angel; R\"atty\"a, Jouni: Weighted 
                        Bergman spaces induced by rapidly increasing weights. 
                        Mem. Amer. Math. Soc. 227 (2014), no. 1066, vi+124pp. 

\bibitem{PelRat2020}    Pel\'aez, Jos\'e \'Angel and R\"atty\"a, Jouni: 
                        Bergman projection induced by radial weight, https://
                        arxiv.org/pdf/1902.09837.pdf.

\bibitem{PelRatEmb}     Pel\'aez, Jos\'e \'Angel and R\"atty\"a, Jouni:
                        Embedding theorems for Bergman spaces via harmonic 
                        analysis, Math. Ann. 362 (2015), no. 1-2, 205--239.

\bibitem{PelRatSchatten}    Pel\'aez, Jos\'e \'Angel; R\"atty\"a, Jouni:
                            Trace class criteria for Toeplitz and composition 
                            operators on small Bergman spaces, Adv.~Math. 293 
                            (2016), 606--643.

\bibitem{PelSum14}          Pel\'aez, Jos\'e \'Angel: Small weighted Bergman 
                            spaces, Proceedings of the summer school in 
                            complex and harmonic analysis, and related 
                            topics, (2016).

\bibitem{JJK}               Pel\'aez, Jos\'e \'Angel; R\"atty\"a, Jouni; 
                            Sierra, Kian: Berezin transform and Toeplitz 
                            operators on Bergman spaces induced by regular 
                            weights. J. Geom. Anal. 28 (2018), no. 1,656-687.

\bibitem{JAJ}               Pel\'aez, Jos\'e \'Angel; Per\"al\"a, Antti and 
                            R\"atty\"a, Jouni: operators induced by radial 
                            Bekoll\'e - Bonami weights on Bergman spaces, 
                            preprint, https://arxiv.org/abs/1806.09854.                            

\bibitem{SE}            Saukko, Erno: Difference of composition operators 
                        between standard weighted Bergman spaces. J. Math. 
                        Anal. Appl. 381 (2011), no. 2, 789--798.

\bibitem{SEO}           Saukko, Erno: An application of atomic decomposition 
                        in Bergman spaces to the study of differences of                     
                        composition operators. J. Funct. Anal. 262 (2012), 
                        no. 9, 3872--3890.

\bibitem{SJH1}          Shapiro, Joel H.: The essential norm of a composition 
                        operator. Ann. of Math. (2) 125 (1987), no. 2, 375--
                        404. 

\bibitem{SJH2}          Shapiro, Joel H.: Composition operators and classical 
                        function theory. Universitext: Tracts in Mathematics. 
                        Springer-Verlag, New York, 1993. xvi+223 pp. 

\bibitem{SS}            Shapiro, Joel H.; Sundberg, Carl: Isolation amongst 
                        the composition operators. Pacific J. Math. 145 
                        (1990), no. 1, 117--152.

\bibitem{SW}            Smith, Wayne: Composition operators between Bergman and 
                        Hardy spaces. Trans. Amer. Math. Soc. 348 (1996), no. 6, 
                        2331--2348. 

\bibitem{SWY}           Smith, Wayne; Yang, Liming: Composition operators that 
                        improve integrability on weighted Bergman spaces. Proc. 
                        Amer. Math. Soc. 126 (1998), no. 2, 411--420.
\end{thebibliography}
\end{document}